\documentclass[10pt]{amsart}

\usepackage{amssymb}
\usepackage[all]{xy}
\usepackage{enumerate}

\DeclareMathOperator{\Char}{{char}}

\newcommand{\rar}[1]{\stackrel{#1}{\longrightarrow}}

\newcommand{\bA}{{\mathbb A}}

\newcommand{\bF}{{\mathbb F}}

\newcommand{\bP}{{\mathbb P}}

\newcommand{\bZ}{{\mathbb Z}}

\newcommand{\cA}{{\mathcal A}}

\newcommand{\cE}{{\mathcal E}}

\newcommand{\cO}{{\mathcal O}}
\newcommand{\cP}{{\mathcal P}}
\newcommand{\cQ}{{\mathcal Q}}
\newcommand{\cR}{{\mathcal R}}

\newcommand{\cW}{{\mathcal W}}

\newcommand{\cZ}{{\mathcal Z}}

\newcommand{\Zb}{\overline{Z}}

\newcommand{\td}{\tilde{d}}

\newcommand{\tv}{\tilde{v}}
\newcommand{\tw}{\widetilde{w}}
\newcommand{\tx}{\widetilde{x}}
\newcommand{\ty}{\widetilde{y}}
\newcommand{\tz}{\widetilde{z}}

\newcommand{\rx}{\widetilde{x^{[2]}}}
\newcommand{\ry}{\widetilde{y^{[2]}}}
\newcommand{\rz}{\widetilde{z^{[2]}}}

\newcommand{\nc}{\newcommand}

\nc\wh{\widehat}

\nc\on{\operatorname}

\nc\Gr{\on{Gr}}

\nc\Fl{\on{Fl}}

\newcommand{\limto}{{\displaystyle\lim_{\longrightarrow}}}
\newcommand{\rightlim}{\mathop{\limto}}

\newcommand{\leftlim}{\mathop{\displaystyle\lim_{\longleftarrow}}}
\newcommand{\limfromn}{\leftlim\limits_{\raise3pt\hbox{$n$}}}
\newcommand{\limton}{\rightlim\limits_{\raise3pt\hbox{$n$}}}

\newcommand{\rightlimit}[1]{\mathop{\lim\limits_{\longrightarrow}}\limits%
                    _{\raise3pt\hbox{$\scriptstyle #1$}}}

\newcommand{\leftlimit}[1]{\mathop{\lim\limits_{\longleftarrow}}\limits%
                    _{\raise3pt\hbox{$\scriptstyle #1$}}}

\newcommand{\epi}{\twoheadrightarrow}
\newcommand{\iso}{\buildrel{\sim}\over{\longrightarrow}}
\newcommand{\mono}{\hookrightarrow}

\DeclareMathOperator{\Ker}{{Ker}} 
\DeclareMathOperator{\im}{{Im}}

\DeclareMathOperator{\Spec}{{Spec}}

\DeclareMathOperator{\cone}{{cone}}

\makeatletter

\newcommand{\Rmnum}[1]{\expandafter\@slowromancap\romannumeral #1@}
\makeatother

\newtheorem{Th}{Theorem}
\newtheorem{pr}{Proposition}[section]
\newtheorem{lm}[pr]{Lemma}

\newtheorem{cor}[pr]{Corollary}

\theoremstyle{definition}

\newtheorem{rem}[pr]{Remark}
\newtheorem{cl}[pr]{Claim}

\numberwithin{equation}{section}

\begin{document}

\title[Triangulated endofunctors of the derived category of coherent sheaves which do not admit DG liftings]
{Triangulated endofunctors of the derived category of coherent sheaves which do not admit DG liftings}

\author{Vadim Vologodsky}

\address
{National Research University ``Higher School of Economics'' and  the University of Oregon}
\email{vvologod@uoregon.edu}






\begin{abstract} 
In \cite{rv}, Rizzardo and Van den Bergh constructed an example of a triangulated functor between the derived categories of coherent sheaves on smooth projective varieties over a field $k$ of characteristic $0$ which is not of the  Fourier-Mukai type.  The purpose of this note is to show that if $\Char k =p$ then there are very simple examples of such functors. Namely, for a smooth projective $Y$ over $\bZ_p$ with the special fiber $i: X\mono Y$, we consider the functor $L i^*   \circ  i_*: D^b(X) \to  D^b(X)$ from the  derived categories of coherent sheaves on $X$ to itself. We show that if $Y$ is a flag variety which is not isomorphic to $\bP^1$ then $L i^*   \circ  i_*$ is not  of the  Fourier-Mukai type. Note that by a theorem of Toen (\cite{t}, Theorem 8.15) the latter assertion is equivalent to saying that $L i^*   \circ  i_*$  does not admit a lifting to a $\bF_p$-linear DG quasi-functor $D^b_{dg}(X) \to  D^b_{dg}(X)$, where $D^b_{dg}(X)$ is a (unique) DG enhancement of $D^b(X)$. However, essentially by definition, $L i^*   \circ  i_*$ lifts to a 
$\bZ_p$-linear DG quasi-functor.

 \end{abstract}

\maketitle

Given smooth proper schemes $X_1, X_2$ over a field $k$ and an object  $E\in D^b(X_1\times X_2)$ of the bounded derived category of coherent sheaves on $X_1\times X_2$
define a triangulated functor
\begin{equation}
\Phi_E:   D^b(X_1) \to  D^b(X_2)
\end{equation}
sending a bounded complex $M$ of coherent sheaves on $X_1$ to  
$ Rp_{2 *}(E\stackrel{L}{\otimes} p^*_1 M)$, where $p_i: X_1\times X_2 \to X_i$ are the projections.
Recall that a triangulated functor $D^b(X_1) \to  D^b(X_2)$ is said to be of the  Fourier-Mukai type if it is isomorphic to $\Phi_E$ for some $E$.

Let $Y$  be a smooth projective scheme  over $\Spec \bZ_p$ and let $X$ be its special fiber,   $i: X \mono Y$ the closed  embedding.
 Consider the triangulated  functor $G: D^b(X) \to  D^b(X)$  given by the formula
               $$                     G =  L i^*   \circ  i_*            $$
We shall see that in general $G$ is not   of  the Fourier-Mukai type. 
\begin{Th} Let  $Z$   a smooth projective scheme over $\Spec \bZ_p$, $Y= Z\times Z$, $X= Y\times _{\Spec \bZ_p} \Spec \bF_p$ . Assume that
\begin{enumerate}
\item The Frobenius morphism $Fr: \Zb \to \Zb$, where $\Zb= Z \times \Spec \bF_p$, does not lift modulo $p^2$.
\item $H^1(X, T_X)=0$, where $T_X$ is the tangent sheaf on $X$.
\end{enumerate} 
 Then  $G =  L i^*   \circ  i_*:  D^b(X) \to  D^b(X)$ 
 is not   of the  Fourier-Mukai type. 
\end{Th}
For example,  let $GL_n$ be the general linear group over $\Spec \bZ_p$, $B\subset GL_n$ a Borel subgroup.  Then,  by Theorem 6 from \cite{btlm}, for any $n>2$, the flag variety $Z= GL_n/B$
satisfies the first assumption of the Theorem {\it i.e.}, the Frobenius $Fr: \Zb \to \Zb$ does not lift on $Z \times \Spec \bZ/p^2 \bZ$. By (\cite{klt}, Theorem 2), we have that $H^1(\Zb, T_{\Zb})= H^1(\Zb, \cO_{\Zb})=0$.
It follows that $H^1(X, T_X)=0$.  Hence, by the Theorem, for $n>2$, $G: D^b(X) \to  D^b(X)$  is not  of  the Fourier-Mukai type.
  \begin{proof}
Assume the contrary and let $E\in D^b(X\times X)$ be the  Fourier-Mukai kernel. By definition, for every $M\in  D^b(X)$ we have a functorial isomorphism
\begin{equation}\label{pr1}
G(M)\iso Rp_{2 *}(E\stackrel{L}{\otimes} p^*_1 M).
\end{equation}
By the projection formula (\cite{h}, Chapter \rm{II}, Prop. 5.6) we have that
$$i_*  \circ  L i^*   \circ  i_*(M) \iso  i_*(M) \stackrel{L}{\otimes}   i_*(\cO_X) \iso  i_*(M) \otimes   (\cO_Y \rar{p} \cO_Y)\iso  i_*(M) \oplus  i_*(M)[1]$$
In particular, if $M $ is a coherent sheaf then $\underline{H}^i(G(M)) \simeq M$ for $i=0, -1$ and  $\underline{H}^i(G(M))=0$ otherwise. 
Applying this observation and formula (\ref{pr1}) to skyscraper sheaves, $M= \delta_x$, $x\in X(\overline \bF_p)$, we conclude that  the coherent sheaves $\underline{H}^i(E)$ are set theoretically supported on the diagonal $\Delta_X \subset X\times X$. Applying the same formulas to  $M= \cO_X$ we see that
$ p_{2 *}(\underline{H}^i(E))=\cO_X$  for $i=0, -1$ and $ p_{2 *}(\underline{H}^i(E))=0$ otherwise. In fact,  every coherent sheaf $F$ on $X\times X$ which is set theoretically 
supported on the diagonal and such that $p_{2 *} F = \cO_X$ is isomorphic to $\cO_{\Delta_X}$. It follows that $\underline{H}^0(E)= \underline{H}^{-1}(E)= \cO_{\Delta_X}$. 
 In the other words, $E$ fits into an exact triangle in  $ D^b(X\times X)$
\begin{equation}\label{pr5}
 \cO_{\Delta_X}[1] \rar{\alpha}  E \rar{}  \cO_{\Delta_X} \rar{\beta}  \cO_{\Delta_X}[2]
 \end{equation}
for some $\beta \in Ext^2_{\cO_{X\times X}}( \cO_{\Delta_X}, \cO_{\Delta_X}).$
We wish to show that the second assumption in the Theorem implies that $\beta = 0$, while the first one implies that $\beta \ne 0$.
For every  $M \in  D^b(X)$, (\ref{pr5}) gives rise to an exact triangle  
\begin{equation}\label{pr3}
M[1] \rar{\alpha_M} G(M)\rar{} M \rar{\beta_M} M[2]
\end{equation}
 Our main tool is the following result.
\begin{lm} For a coherent sheaf $M$ the following conditions  are equivalent.
\begin{enumerate}
\item $\beta_M=0$.
\item $G(M)\iso M\oplus M[1]$.
\item There exists a morphism $\lambda: G(M) \to M[1]$ such that 
 $\lambda \circ \alpha_M$ is an isomorphism.
\item $M$ admits a lift modulo $p^2$ {\it i.e.}, there is a coherent sheaf $\tilde M$ on $Y$ flat over $\bZ/p^2 \bZ$ such that $i^*\tilde M \simeq M$.
\end{enumerate}
\end{lm}
\begin{proof}
The equivalence of (1), (2) and (3) is immediate. Let us check that (3) is equivalent to (4). A morphism $\lambda: G(M) \to M[1]$ gives rise by adjunction a morphism $\gamma: i_*M \to i_* M[1]$. Note that $\tilde M = \cone \gamma [1]$ is a coherent sheaf on $Y$ which is an extension of $i_*M$ by itself.
It suffices to prove that $\lambda \circ \alpha_M: M[1] \to M[1]$ is an isomorphism if and only if $\tilde M$ is flat over $\bZ/p^2 \bZ$. 
Indeed, from the exact triangle 
$$ Li^*i_* M \to Li^* (\tilde M)  \to  Li^*i_*M \to  Li^*i_* M[1]$$
we get a long exact sequence of the cohomology sheaves
$$     0\to  M\to L_1i^* (\tilde M) \to M\rar{\lambda \circ \alpha_M}  M\to  i^* (\tilde M) \to M \to 0 $$
Thus $\lambda \circ \alpha_M$ is an isomorphism if and only if in the exact sequence 
$$0\to i_*M \to \tilde M \to i_*M \to 0$$
the image of second map is the kernel of the multiplication by $p$ on $\tilde M$ and also the image of this map. The latter is equivalent to flatness of $\tilde M$  over $\bZ/p^2 \bZ$. 
\end{proof}
We have a spectral sequence converging to $Ext^*_{\cO_{X\times X}}( \cO_{\Delta_X}, \cO_{\Delta_X}) $ whose second page is 
$ H^*(X,  \cE xt^* _{\cO_{X\times X}}( \cO_{\Delta_X}, \cO_{\Delta_X}))$. In particular, we have a homomorphism
$$Ext^2_{\cO_{X\times X}}( \cO_{\Delta_X}, \cO_{\Delta_X}) \to H^0(X,  \cE xt^2 _{\cO_{X\times X}}( \cO_{\Delta_X}, \cO_{\Delta_X})) \iso H^0(X, \wedge^2 T_X).$$
Let us check that the image $\mu$ of $\beta$ under this map
is $0$. To do this we apply the Lemma to skyscraper sheaves $\delta_x$, where $x$ runs over closed points of $X$. On the one hand, the evaluation of the bivector field $\mu$ at $x$
is equal to the class of $\beta_{\delta_x}$ in $Ext^2_{\cO_{X}}( \delta_x, \delta_x) \iso \wedge^2 T_{x, X}.$ 
On the other hand, by the Lemma, $\beta_{\delta_x}=0$ since $\delta_x$ is liftable modulo $p^2$.
  Next, the assumption that $H^1(X, T_X)=0$ implies that $\beta$   lies in the image of the map
  \begin{equation}\label{pr4}
v:  H^2(X, \cO_X) \iso H^2 (X,  \cE xt^0 _{\cO_{X\times X}}( \cO_{\Delta_X}, \cO_{\Delta_X}) ) \to Ext^2_{\cO_{X\times X}}( \cO_{\Delta_X}, \cO_{\Delta_X}) .
  \end{equation}
  The map (\ref{pr4}) has a left inverse $u: Ext^2_{\cO_{X\times X}}( \cO_{\Delta_X}, \cO_{\Delta_X})  \to  H^2(X, \cO_X)$ which takes $\beta$ to $\beta_{\cO_X}$. But, by the Lemma, the later class is equal to $0$ since $\cO_X$ is liftable modulo $p^2$. It follows that $\beta$ is $0$.
  
On the other hand, let  $\Gamma \subset X= \Zb \times \Zb$ be the graph of the Frobenius morphism $Fr: \Zb \to \Zb$ and $\cO_{\Gamma}$ the structure sheaf of $\Gamma$ viewed as a coherent sheaf on $X$. Then, by our first assumption, the sheaf  $\cO_{\Gamma}$ is not liftable modulo $p^2$. Hence, by the Lemma, $\beta_{\cO_{\Gamma}}$ is not $0$. This contradiction 
  completes the proof. 
\end{proof}

{\bf Acknowledgements.} I would like to thank  Alberto Canonaco and Paolo Stellari: their interest prompted writing this note. Also, I am grateful to Alexander Samokhin for stimulating discussions and references.


\begin{thebibliography}{BFM}

\bibitem[BTLM]{btlm} 
A. Buch, J. F. Thomsen, N. Lauritzen, and V. Mehta, {\it The Frobenius morphism on a toric variety,}  Tohoku Math. J. (2), Volume 49, Number 3 (1997), 355-366.




\bibitem[CV]{cv}  D. Calaque, M. Van den Bergh, {\it Hochschild cohomology and Atiyah classes, }  Adv. Math. 224 (2010), no. 5, 1839--1889. 


\bibitem[H]{h}   R. Hartshorne, {\it Resudues and duality,} LNM 20 (1966).

\bibitem[KLT]{klt}  S. Kumar, N. Lauritzen, J.F.Thomsen, Frobenius splitting of cotangent bundles of flag varieties, Invent. Math.
136 (1999), pp. 603-62

\bibitem[RV]{rv}  A. Rizzardo, M. Van den Bergh,  {\it An example of a non-Fourier?Mukai functor between derived categories of
coherent sheaves,}  arXiv:1410.4039.
   
\bibitem[T]{t}  B. To\"en, The homotopy theory of dg-categories and derived Morita theory, Invent. Math. 167 (2007),   
   
\end{thebibliography}
\end{document}